\documentclass[english,10pt]{amsart}

\usepackage{amsfonts,amssymb,amsmath,amsgen,amsthm}
\usepackage{hyperref}

\def\R{\mathbb R}

\def\t{\tanh (2\sigma t)}
\def\s{\sinh(2\sigma t)}
\def\ts{\tanh (2\sigma s)}
\def\cs{\cosh(2\sigma s)}
\def\ss{\sinh(2\sigma s)}

\def\tbis{\tan(2\sigma t)}
\def\cbis{\cos(2\sigma t)}
\def\sbis{\sin(2\sigma t)}
\def\tsbis{\tan (2\sigma s)}
\def\csbis{\cos(2\sigma s)}
\def\ssbis{\sin(2\sigma s)}

\theoremstyle{plain}
\newtheorem{theo}{Theorem} [section]

\newtheorem{prop}[theo]{Proposition}
\theoremstyle{remark}
\newtheorem{rem}[theo]{Remark}

\def\({\left(}
\def\){\right)}
\def\<{\left\langle}
\def\>{\right\rangle}

\def\Tend#1#2{\mathop{\longrightarrow}\limits_{#1\rightarrow#2}}

\numberwithin{equation}{section}

\begin{document}

\title[Replicator-mutator equations with quadratic fitness]{Replicator-mutator equations with quadratic fitness}

\author[M. Alfaro]{Matthieu Alfaro}
\address{CNRS \& Univ. Montpellier\\
    IMAG\\ CC~051\\ 34095 Montpellier\\France}
\email{matthieu.alfaro@umontpellier.fr}
\author[R. Carles]{R\'emi Carles}
\email{remi.carles@math.cnrs.fr}

\begin{abstract} 
This work completes our previous analysis on models
arising in evolutionary genetics. We consider the so-called
replicator-mutator equation, when the fitness is quadratic. This
equation is a heat equation with a harmonic potential, plus a
specific nonlocal
term. We give an explicit formula for the
solution, thanks to which we  prove that when the fitness is
non-positive (harmonic potential),  solutions converge
to a universal stationary Gaussian for large time, whereas when the
fitness is non-negative (inverted harmonic potential),  
solutions always become extinct in finite time.  
\end{abstract}

\maketitle

\section{Introduction}

This  note is concerned with {\it replicator-mutator} equations, that is nonlocal
reaction diffusion problems of the form
\begin{equation*}\label{eq-bio-poids}
\partial _t U={\sigma _0}^2\partial _{xx}U+\mu _0 \left(f(x)-\int _\R
  f(x)U(t,x)\,dx\right)U, \quad t>0,\; x\in \R, 
\end{equation*}
where $\sigma _0>0$ and $\mu _0>0$ are parameters, and when either
$f(x)=-x^{2}$ or $f(x)=x^{2}$. In order to simplify the  presentation
of the results, and before going into more details, we use the
rescaling 
$$
u(t,x):=U\left(\frac t{\mu _0},x\right), \quad
\sigma:=\frac{\sigma _0}{\sqrt \mu _0},
$$
and therefore consider
\begin{equation}\label{eq-bio}
\partial _t u=\sigma ^2 \partial _{xx}u+(f(x)-\overline f (t))u, \quad
t>0,\; x\in \R, 
\end{equation}
where the nonlocal term is given by
\begin{equation}\label{premier-moment}
\overline f (t):=\int _\R f(x)u(t,x)\,dx.
\end{equation}
Equation \eqref{eq-bio} is always supplemented with an initial
condition $u_0\geq 0$ with mass $\int _\R u_0=1$, so that the mass is
\emph{formally} conserved for later times. Indeed, integrating
\eqref{eq-bio} over $x\in \R$, we find that $m(t):=\int_{\R}u(t,x)dx$
satisfies 
\begin{equation*}
  \frac{dm}{dt} = \overline f(t)\(1-m(t)\),\quad m(0)=1,
\end{equation*}
hence $m(t)=1$ so long as $\overline f$ is
integrable. 

In the context of evolutionary genetics, Equation \eqref{eq-bio}
was introduced by Tsimring  et al. \cite{TLK96}, where they
propose a mean-field theory for the evolution of RNA virus
populations on a phenotypic trait space.  In this context, $u(t,x)$ is the
density of a population (at time $t$ and per unit of phenotypic
trait) on a one-dimensional phenotypic trait space, and $f(x)$
represents the fitness of an individual with trait value $x$ in a population 
which is at state $u(t,x)$.
The nonlocal term $\overline{f}(t)$ represents the mean fitness at time $t$. We refer to \cite{Alf-Car} for more  references on the biological background of \eqref{eq-bio}.

\subsection{The  case $f(x)=x$}
This case can be seen as a parabolic counterpart of the Schr\"odinger
equation with a Stark potential; in the context of quantum mechanics,
this potential corresponds to a constant electric field or to
gravity  (see e.g. \cite{Thirring}).
In the case of \eqref{eq-bio}, a family of self similar Gaussian
solutions has been constructed in 
\cite{B14}. Then this case has been completely studied in
\cite{Alf-Car}. It turns out that not only traveling pulses are
changing 
sign, but also extinction in finite time occurs for initial
data with  ``not very light tails'' (data which do not decay very fast
on the right). This, in particular,
contradicts the formal conservation of the mass observed in
\eqref{eq-bio} and evoked above. Roughly speaking, the nonlocal term
of the equation $\int _\R xu(t,x)dx$ becomes infinite and  the
equation becomes meaningless. On the other hand, for initial data with
 ``very light tails'' (they decay sufficiently fast on the right), the
 solution is defined for all times 
$t\geq 0$ and  
is escaping to the right by accelerating and flattening as $t\to
\infty$. More precisely, the long time behaviour is (when $\sigma=1$)
a Gaussian centered at $x(t)=t^2$ (acceleration) and of maximal height
$1/\sqrt{4\pi t}$ (flattening effect). In other words, extinction
occurs at $t=\infty$ in this situation. This is like in the case of the linear
heat equation, up to the fact that the center of the asymptotic
Gaussian is given by $x(t)$, which undergoes some acceleration which is
reminiscent of the effect of gravity.  Notice some links of this
acceleration phenomena 
with some aspects of the so-called {\it dynamics of the fittest
trait} (see \cite{DL96}, \cite{DJMP05}, \cite{MPW12} and the
references therein) which, in some cases, escape to infinity for
large times \cite{D04}, \cite{P07}.

\subsection{The quadratic cases}
To prove the above  results in \cite{Alf-Car}, we used a change
of unknown 
function based on the Avron--Herbst formula for the Schr\"odinger
equation, and  showed that \eqref{eq-bio} is equivalent to the heat 
equation. We could then compute its solution explicitly. Without those
exact computations, the understanding of \eqref{eq-bio} seems far from
obvious, and in particular the role of the decay on the right of the
initial data. In \cite{Alf-Car} we  indicated that similar  computations
could also be performed in the cases $f(x)=\pm x^2$, thanks to the
generalized lens transform of the Schr\"odinger equation, but without
giving any detail. The goal of the present work is to fill this gap,
by giving the full details and results for these two cases. 
\smallbreak

Our motivation  is twofold. First, very recently, replicator-mutator
equations (or related problems) with quadratic fitness have attracted
a lot of attention: let us mention the works \cite{LoMiPe11},
\cite{Cal-Cud-Des-Rao}, \cite{Chi-Lor-Des-Hug}, 
\cite{Mar-Roq}, \cite{Gil-Ham-Mar-Roq}, \cite{Ver}  and the references therein. In
particular, Chisholm et al. \cite{Chi-Lor-Des-Hug} study --- among
other things --- the long time behaviour of the nonlocal term of an
equation very close to \eqref{eq-bio}, in the case $f(x)=-x^2$, with
compactly supported initial data. In Section \ref{s:moinsquadratic} we
completely solve \eqref{eq-bio} for any initial data, and can
therefore study the long time behaviour not only of the nonlocal term
$\overline f(t)$ but also of the full profile $u(t,x)$. The second
reason is  that the obtained behaviours are varied and
interesting, bringing precious information in the dynamical study of
partial differential equations. To give a preview of this, we now state two theorems which are direct consequences of
the more detailed results of Section \ref{s:moinsquadratic} ($f(x)=-x^2$) and Section \ref{s:plusquadratic} ($f(x)=x^2$).

 In the case $f(x)=-x^{2}$, solutions
tend at large time 
to a universal stationary Gaussian. Denote 
\begin{equation*}
  \mathcal M _2(\R):=\left\{g\in L^1(\R),\quad \int_\R x^2 |g(x)|dx<\infty\right\}.
\end{equation*}
\begin{theo}[Case $f(x)=-x^2$]\label{theo:1}
  Let $u_0\ge 0$, with  $\int _\R
  u_0=1$. Then \eqref{eq-bio}, with  initial datum $u_0$, has a unique solution $u\in
  C(\R_+;L^1(\R))\cap L^1_{\rm loc}((0,\infty);\mathcal M _2(\R))$. It  satisfies
\begin{equation*}
\sup _{x\in \R}\; \left |u(t,x)-\varphi(x)\right|\Tend t \infty
0,\quad \text{where }\varphi(x):=\frac
1{\sqrt{2\pi\sigma}}e^{-x^2/(2\sigma)}.
\end{equation*}
\end{theo}
The above result shows that the presence of the quadratic fitness
compensates the diffusive mechanism of the heat equation, since we
recall that for $u_0\ge 0$ with $\int (1+|x|)u_0(x)dx<\infty$,
\begin{equation*}
  e^{t\partial_{xx}}u_0(x) =\frac{\|u_0\|_{L^1(\R)}}{\sqrt{4\pi
      t}}e^{-x^2/(4t)}+o(1),\quad \text{in }L^\infty(\R), \text{ as }t\to \infty.
\end{equation*}
We also emphasize that in Theorem~\ref{theo:1}, we do not assume that
$u_0$ has two momenta in $L^1(\R)$, $u_0\in \mathcal M _2(\R)$: this property
is satisfied by the solution instantaneously, as we will see in
Section~\ref{s:moinsquadratic}. 
  
On the other hand, in  the case
$f(x)=x^2$, extinction in finite time always occurs.
\begin{theo}[Case $f(x)=x^2$]\label{theo:2}
  Let $u_0\ge 0$, with  $\int _\R u_0=1$. Then the solution to
  \eqref{eq-bio}, with   initial datum $u_0$, 
  becomes extinct in finite time:
  \begin{equation*}
    \exists T\in \left[0,\frac{\pi}{4\sigma}\right],\quad
    u(t,x) =0, \quad 
    \forall t>T,\ \forall x\in \R. 
  \end{equation*}
\end{theo}
As in \cite{Alf-Car}, the solution may become extinct
\emph{instantaneously} ($T=0$), that is, \eqref{eq-bio} has no
non-trivial solution, if the initial datum has too little decay at 
infinity.

\subsection{Heat vs. Schr\"odinger}
As mentioned above, the present results, as well as those established
in \cite{Alf-Car}, stem from explicit formulas discovered in the
context of Schr\"odinger equations; see
\cite{Niederer73,Niederer74}, \cite{Thirring}.  However, we have to emphasize
several differences between \eqref{eq-bio} and its Schr\"odinger
analogue,
\begin{equation}\label{eq-schrod}
i\partial _t u=\sigma ^2 \partial _{xx}u+(f(x)-\overline f (t))u, \quad
t\in \R,\; x\in \R.
\end{equation}
The Schr\"odinger equation is of course time reversible. Less obvious is the
way the term $\overline f(t) u$ is handled, according to the
equation one considers. In the Schr\"odinger case \eqref{eq-schrod},
we simply use a gauge transform to get rid of this term: it is
equivalent to consider $u$ solution to \eqref{eq-schrod} or
\begin{equation}\label{eq:gauge}
  v(t,x) = u(t,x) e^{-i\int_0^t\overline f(s)ds},
\end{equation}
which solves
\begin{equation*}
  i\partial _tv =\sigma ^2 \partial _{xx}v+f(x)v,
\end{equation*}
with the same initial datum. If we assume, like in the case of
\eqref{eq-bio}, that $\overline f(t)$ is 
real (which means that it is 
\emph{not} given by \eqref{premier-moment} in this case), the change
of unknown function 
\eqref{eq:gauge} does not alter the dynamics, since $|
v(t,x)|=|u(t,x)|$. On the other hand, the analogous transformation in
the parabolic case has a true effect on the dynamics since, as pointed
out in \cite{Alf-Car}, it becomes
\begin{equation}\label{eq:uv-para}
  v(t,x) = u(t,x) e^{\int_0^t \overline f (s)ds}.
\end{equation}
Multiplying by $f(x)$ and integrating in $x$, we infer
\begin{equation*}
  \int_{\R} f(x)v(t,x)dx = \overline f(t)e^{\int_0^t \overline f
    (s)ds}=\frac{d}{dt}\(e^{\int_0^t \overline f (s)ds} \). 
\end{equation*}
Therefore, \emph{so long as} $\int_0^t \int_{\R} f(x)v(s,x)dxds>-1$, 
\begin{equation*}
  u(t,x) = \frac{v(t,x)}{1+\int_0^t \int_{\R} f(x)v(s,x)dxds}.
\end{equation*}
It is clear that in general, $u$ and $v$ now have different large time
behaviours. 

The last algebraic step to construct explicit solutions for
\eqref{eq-bio} and \eqref{eq-schrod} consists in using the
Avron--Herbst formula when $f(x)=x$, or a (generalized) lens transform
when $f(x)=\pm x^2$. In the case of the standard (quantum) harmonic
oscillator, the solutions to
\begin{equation*}
  i\partial_t v +\partial_{xx} v =\omega^2 x^2v,\quad \text{and}\quad
  i\partial_t w +\partial_{xx} w=0,\quad \text{with }v_{\mid
    t=0}=w_{\mid t=0},
\end{equation*}
are related through the formula
\begin{equation*}
v(t,x) = \frac{1}{\sqrt{\cos (2\omega t)}} w\(\frac{\tan (2\omega
  t)}{2\omega},\frac{x}{\cos (2\omega t)}\)e^{-i\frac{\omega}{2} x^2\tan
  (2\omega t)},\quad |t|<\frac{\pi}{4\omega}.
\end{equation*}
What this formula does not show, since it is limited in time, is that
the solution $v$ is periodic in time, as can be seen for instance by
considering an eigenbasis for the harmonic oscillator
 $ -\partial_{xx}+\omega^2x^2$,
given by Hermite functions. Suppose $\omega=1$ to lighten the
notations: the Hermite functions $(\psi_j)_{j\ge 0}$ form an
orthogonal basis of $L^2(\R)$, 
and 
\begin{equation*}
  -\partial_{xx}\psi_j +x^2\psi_j = (2j+1)\psi_j. 
\end{equation*}
Therefore, if 
\begin{equation*}
  v(0,x) = \sum_{j\ge 0} \alpha_j\psi_j(x),\quad \text{then }v(t,x) =
  \sum_{j\ge 0} \alpha_j\psi_j(x) e^{i(2j+1)t}
\end{equation*}
is obviously $2\pi$-periodic in time. This is in sharp contrast with
the behaviour described in Theorem~\ref{theo:1}. Similarly, the
solutions to 
\begin{equation*}
  i\partial_t v +\partial_{xx} v =-\omega^2 x^2v,\quad
  i\partial_t w +\partial_{xx} w=0,\quad v_{\mid
    t=0}=w_{\mid t=0},
\end{equation*}
are related through the formula (change $\omega$ to $i\omega$ in the
previous formula)
\begin{equation*}
v(t,x) = \frac{1}{\sqrt{\cosh (2\omega t)}} w\(\frac{\tanh (2\omega
  t)}{2\omega},\frac{x}{\cosh (2\omega t)}\)e^{i\frac{\omega}{2} x^2\tanh
  (2\omega t)},\quad t\in \R.
\end{equation*}
This shows that the inverted harmonic potential accelerates the
dispersion ($\|v(t,\cdot)\|_{L^\infty}$ goes to zero exponentially
fast), and the large time profile is given by $w_{\mid t=
  1/(2\omega)}$. Again, this behaviour is completely different from
the one stated in Theorem~\ref{theo:2}. 
\section{The case $f(x)=-x^2$: convergence to a universal
  Gaussian}\label{s:moinsquadratic}

The case $f(x)=-x^2$ can be handled as explained in \cite{Alf-Car}. We
give more details here, and analyze the consequences of the explicit
formula. In particular, for any initial data, the solution is defined for all positive times and converge, at large
time, to a universal stationary Gaussian.

\subsection{Results}

\begin{theo}[The solution explicitly]
\label{th:explicit-sol} Let $u_0\ge 0$, with  $\int _\R u_0=1$. Then
\eqref{eq-bio} with initial datum $u_0$ has a unique solution $u\in
C(\R_+;L^1(\R))\cap L^1_{\rm loc}((0,\infty);\mathcal M _2(\R))$. For all
$t>0$ and $x\in \R$, it is given by
\begin{align}
 u(t,x)&=\frac 1{\sqrt{2\pi\sigma \t}}
 \frac{e^{-\frac \t {2\sigma} x^2}
\displaystyle\int _\R e^{-\frac 1{2\sigma\t} \left(\frac x \cosh(2\sigma t)
-y\right)^2} u_0(y)\,dy} {\displaystyle \int _\R
e^{-\frac \t {2\sigma} y^2}u_0(y)\,dy}\label{formula}\\
&=  \frac 1{\sqrt{2\pi\sigma \t}}
 \frac{
\displaystyle\int _\R e^{-\frac 1{2\sigma \t} \left(x - \frac y
\cosh(2\sigma t)\right)^2} e^{-\frac \t {2\sigma} y^2}u_0(y)\,dy} {\displaystyle
\int _\R e^{-\frac \t {2\sigma} y^2}u_0(y)\,dy}.\label{formula-2}
\end{align}
As a consequence, for all
  $t> 0$, $\overline f(t)$ is given by
\begin{equation}\label{u-bar}
\overline  f(t)=\sigma \t+\frac{1}{(\cosh(2\sigma t))^{2}}\frac{\displaystyle \int _\R
e^{-\frac{\t}{2\sigma}y^{2}} y^{2}\, u_0(y)\,dy}{\displaystyle \int
_\R e^{-\frac{\t}{2\sigma}y^{2}} u_0(y)\,dy}. 
\end{equation}
\end{theo}

%\begin{cor}[The nonlocal term explicitly]\label{cor:u-bar} Under the
 % assumptions of Theorem~\ref{th:explicit-sol}, for all
 % $t> 0$, $\overline f(t)$ is given by
%\begin{equation}\label{u-bar}
%\overline  f(t)=\sigma \t+\frac{1}{(\cosh(2\sigma t))^{2}}\frac{\displaystyle \int _\R
%e^{-\frac{\t}{2\sigma}y^{2}} y^{2}\, u_0(y)\,dy}{\displaystyle \int
%_\R e^{-\frac{\t}{2\sigma}y^{2}} u_0(y)\,dy}. 
%\end{equation}
%\end{cor}

We now investigate the propagation of  Gaussian initial data.

\begin{prop}[Propagation of Gaussian initial
data]\label{prop:prop-gauss} If
\begin{equation}\label{initial-gauss}
u_0(x)=\sqrt{\frac a{2\pi}}e^{-\frac a 2(x-m)^2},\quad a>0, \quad
m\in \R,
\end{equation}
then the solution of \eqref{eq-bio} remains  Gaussian for $t>0$
and is given by
\begin{equation}\label{sol-gauss-cond-ini}
u(t,x)=\sqrt{\frac {a(t)}{2\pi}}e^{-\frac{a(t)}2(x-m(t))^2},
\end{equation} where
\begin{equation}\label{sol-gauss-cond-ini2}
a(t):=\frac {a\sigma+\t}{\sigma(1+a\sigma\t)}, \quad
m(t):=\frac{ma\sigma}{a\sigma\cosh(2\sigma t)+\s}.
\end{equation}
\end{prop}

Since $a(t)\to \frac 1 \sigma$ and $m(t)\to 0$ as $t\to \infty$,
it is easily seen that $u(t,x)\to \varphi(x):=\frac 1{\sqrt{2\pi
\sigma }}e^{-\frac 1{2\sigma}x^2}$ uniformly in $x\in \R$.
This fact is actually true for {\it all} initial data, as stated
in the following theorem, which implies Theorem~\ref{theo:1}.

\begin{theo}[Long time behaviour]\label{th:long-time} Under the
  assumptions of Theorem~\ref{th:explicit-sol},  there exists $C>0$
  independent of time  such that 
\begin{equation}\label{deviation}
\sup _{x\in \R}\; \left |u(t,x)-\psi(t,x)\right|\le \frac C
\s,\quad \forall t\ge 1,
\end{equation}
where
$$
\psi(t,x):= \frac 1{\sqrt{2\pi\sigma\t}}e^{-\frac 1{2\sigma
\t}x^2}$$
satisfies 
$$\psi(t,x)\Tend t \infty \varphi(x)=\frac
1{\sqrt{2\pi\sigma}}e^{-\frac1{2\sigma}x^2},
$$
uniformly in $x\in\R$.
\end{theo}

Observe that $\psi(t,x)$ is nothing but the fundamental
solution obtained by plugging $u_0(y)=\delta _0(y)$, the Dirac mass
at the origin, in \eqref{formula-2}, so the first part of the statement
is the analogue of the convergence result recalled in the
introduction,
\begin{equation*}
 \left\| e^{t\partial_{xx}}u_0(x)-\frac{\|u_0\|_{L^1(\R)}}{\sqrt{4\pi
       t}}e^{-x^2/(4t)}\right\|_{L^\infty(\R)}\le
 \frac{C}{t}\|x u_0\|_{L^1(\R)},
\end{equation*}
and the effect of the fitness $f(x)=-x^2$ is to neutralize diffusive
effects.

\subsection{Proofs}

\begin{proof}[Proof of Theorem \ref{th:explicit-sol}]
As proved in \cite{Alf-Car}, we can reduce \eqref{eq-bio} to the
heat equation by combining two changes of unknown function. First, we have
\begin{equation}\label{u-v}
u(t,x)=\frac {v(t,x)}{1-\displaystyle\int _0 ^t\int _\R
x^2v(s,x)\,dxds},
\end{equation}
where $v(t,x)$ solves the Cauchy problem
$$
\partial _t v = \sigma ^2\partial _{xx}v -x^2v, \quad t>0,\; x\in \R;\quad v(0,x)=u_0(x).
$$
Notice that relation \eqref{u-v} is valid as  long as  $\int _0
^t\int _\R x^2v(s,x)\,dxds<1$. Next, by adapting the so-called lens
transform (\cite{Niederer73}, \cite{CaM3AS}), we have
\begin{equation}\label{v-w}
  v(t,x) =\frac
  1{\sqrt{\cosh(2\sigma t)}}e^{-\frac{\t}{2\sigma}x^2}w\left(\frac{\t}{2\sigma},\frac
  x{\sigma\cosh(2\sigma t)}\right),
\end{equation}
where $w(t,x)$ solves the heat equation
\begin{equation*}
   \partial _t w = \partial _{xx}w, \quad t>0,\; x\in \R;\quad w(0,x) =u_0(\sigma x).
\end{equation*}
Combining \eqref{u-v}, \eqref{v-w} and the integral expression of
$w$ via the heat kernel, we end up with
\begin{equation}\label{formula-ugly}
\begin{aligned}
&u(t,x)= \frac{1}{1-I(t)}\times\\
&\sqrt{\frac \sigma{2\pi}}\frac 1{\sqrt{\s}}e^{-\frac
\t{2\sigma}x^2}\displaystyle\int_ \R e^{-\frac \sigma {2\t}
\left(\frac x{\sigma \cosh(2\sigma t)}-y\right)^2}u_0(\sigma y)\,dy,
\end{aligned}
\end{equation}
where
{\small
\begin{equation*}
I(t):=\int _0^t \int _\R \int _ \R x^2 \sqrt{\frac
\sigma{2\pi}}\frac 1{\sqrt{\ss}}e^{-\frac \ts{2\sigma}x^2}
e^{-\frac \sigma {2\ts} \left(\frac x{\sigma
\cs}-y\right)^2}u_0(\sigma y)\,dydxds.
\end{equation*}
}
Let us compute $I(t)$. Using Fubini's theorem, we first compute
the integral with respect to $x$. Using elementary algebra
(canonical form in particular), we get
\begin{align}
&\int _ \R x^2 e^{-\frac \ts{2\sigma}x^2} e^{-\frac \sigma {2\ts}
\left(\frac x{\sigma
\cs}-y\right)^2}\,dx\nonumber\\
&= e^{-\frac{\sigma \ts}2 y^2} \int _\R x^2e^{-\frac 1 {2\sigma \ts}\left(x-\frac {\sigma y }\cs\right)^2}dx \nonumber\\
&
= e^{-\frac{\sigma \ts}2 y^2} \sqrt{ 2 \pi \sigma
\ts}\left(\sigma \ts +\frac{\sigma ^2 y^2}{(\cs) ^2}\right), \label{ici}
\end{align}
where we have used the straightforward formula $\int _\R z^2e^{-\frac a 2
(z-\theta)^2}dz=\frac 1 a \sqrt {\frac {2 \pi}a}+\theta ^2
\sqrt{\frac{2\pi}a}$. Next, we pursue the computation of $I(t)$
and, integrating with respect to $s$, we find
\begin{align*}
&\int _0 ^t \sqrt{\frac \sigma{2\pi}}\frac
1{\sqrt{\ss}}e^{-\frac{\sigma \ts}2 y^2} \sqrt{ 2 \pi \sigma
\ts}\left(\sigma \ts +\frac{\sigma ^2 y^2}{(\cs) ^2}\right)\,ds\\
&=\sigma \int _0 ^t e^{-\frac 1 2 \ln (\cs)}e^{-\frac{\sigma
\ts}2 y^2}\left(\sigma \ts +\frac{\sigma ^2 y^2}{(\cs)
^2}\right)\,ds\\
&=\sigma \int _0 ^t \frac d{ds} \left(-e^{-\frac 1 2 \ln
(\cs)}e^{-\frac{\sigma \ts}2 y^2} \right)\,ds\\
&=\sigma\left(1-\frac 1 {\sqrt {\cosh(2\sigma t)}} e^{-\frac{\sigma \t}2
y^2}\right).
\end{align*}
Finally, we integrate with respect to $y$ and, using $\int_\R  u_0=1$,
get
\begin{align*}
I(t)&=\int _\R \sigma \left(1-\frac 1 {\sqrt {\cosh(2\sigma t)}}
e^{-\frac{\sigma \t}2 y^2}\right)u_0(\sigma y)\,dy\\
&=1-\frac 1{\sqrt {\cosh(2\sigma t)}} \int _\R e^{-\frac{\t}{2\sigma}z^2}
u_0(z)\,dz.
\end{align*}
Plugging this in the denominator of \eqref{formula-ugly}, and
using the change of variable $z=\sigma y$ in the numerator of
\eqref{formula-ugly}, we get \eqref{formula}, from which
\eqref{formula-2} easily follows. 
Using
  \eqref{formula-2}, Fubini theorem and the same computation as in
  \eqref{ici}, we obtain \eqref{u-bar}. 
The
solution thus obtained satisfies $u\in C(\R_+;L^1(\R))\cap L^1_{\rm
  loc}((0,\infty);\mathcal M _2(\R))$.  
Uniqueness for such a solution stems from the transformations that we
have used, which require exactly this regularity (see also
\eqref{eq:uv-para}). Theorem~\ref{th:explicit-sol} 
is proved.
\end{proof}

\begin{proof}[Proof of Proposition \ref{prop:prop-gauss}] We plug
the Gaussian data \eqref{initial-gauss} into formula
\eqref{formula}. Using elementary algebra (canonical form), we
first compute
\begin{align*}
&\int _\R e^{-\frac \t {2\sigma} y^2}u_0(y)\,dy\\
&=\sqrt{\frac
a{2\pi}} e^{-\frac {m^2}2\left(a-\frac {a^2}{\frac \t
\sigma+a}\right)}\int _\R e^{-\frac 1 2 \left(\frac
\t \sigma +a\right)\left(y-\frac{am}{\frac \t \sigma +a}\right)^2}\,dy\\
&= \sqrt{\frac{a\sigma}{\t+a\sigma}}e^{-\frac {m^2}2
\frac{a\t}{\t+a\sigma}}.
\end{align*}
Some tedious but similar manipulations involving canonical form
imply
\begin{align*}
&\int _\R e^{-\frac
1{2\sigma\t} \left(\frac x \cosh(2\sigma t) -y\right)^2} u_0(y)\,dy\\
&=\sqrt{\frac a{2\pi}}e^{-\frac a {(\cosh(2\sigma t))^2(1+a\sigma\t)}\frac
{x^2}2}e^{\frac {am}{\cosh(2\sigma t)(1+a\sigma \t)}x}e^{-\frac a {2(1+a\sigma
\t)} m^2}\\
&  \quad\times  \int _\R e^{-\frac{1+a\sigma \t}{2\sigma
  \t}\left[y-\frac{\sigma \t}{1+a\sigma \t}\left(\frac x{\sigma \t
  \cosh(2\sigma t)}+am\right)\right]^2}\,dy\\ 
&= \sqrt{\frac {a\sigma \t}{1+a\sigma \t}}e^{-\frac a
{(\cosh(2\sigma t))^2(1+a\sigma\t)}\frac {x^2}2}e^{\frac {am}{\cosh(2\sigma t)(1+a\sigma
\t)}x}e^{-\frac a {2(1+a\sigma \t)} m^2}.
\end{align*}
Putting all together into \eqref{formula}, we arrive, after 
computations involving hyperbolic functions, at the desired
formulas \eqref{sol-gauss-cond-ini} and
\eqref{sol-gauss-cond-ini2}.
\end{proof}

\begin{proof}[Proof of Theorem \ref{th:long-time}] Since
$\psi(t,x)$ is nothing but the fundamental solution arising
from \eqref{formula-2} with $u_0(y)=\delta_0(y)$, we write
$$
\psi(t,x)= \frac 1{\sqrt{2\pi\sigma \t}}
 \frac{\displaystyle\int _\R e^{-\frac 1{2\sigma \t} x^2} e^{-\frac \t {2\sigma} y^2}u_0(y)\,dy} {\displaystyle
\int _\R e^{-\frac \t {2\sigma} y^2}u_0(y)\,dy},
$$
so that the deviation from this fundamental solution is given by
\begin{align}
&(u(t,x)-\psi(t,x))\sqrt{2\pi \sigma \t}\nonumber \\
&=\frac{\displaystyle\int _\R \(e^{-\frac 1{2\sigma \t} \left(x -
\frac y \cosh(2\sigma t)\right)^2}-e^{-\frac 1{2\sigma \t} x^2}\) e^{-\frac \t
{2\sigma} y^2}u_0(y)\,dy} {\displaystyle \int _\R e^{-\frac \t
{2\sigma} y^2}u_0(y)\,dy}.\label{diff}
\end{align}
Define $G(z):=e^{-z^2/(2\sigma)}$. It follows from the
mean value theorem that
$$
|u(t,x)-\psi(t,x)|\sqrt{2\pi\sigma \t}\leq \Vert G ' \Vert
_\infty \frac{\displaystyle\int _\R \frac {|y|}{\cosh(2\sigma t) \sqrt {\t}}
e^{-\frac \t {2\sigma} y^2}u_0(y)\,dy}{\displaystyle \int _\R
e^{-\frac \t {2\sigma} y^2}u_0(y)\,dy},
$$
which in turn implies
\begin{align*}
|u(t,x)-\psi(t,x)|&\leq \frac {\Vert G'\Vert _\infty }
{\sqrt{2\pi\sigma}\s}\frac{\displaystyle\int _\R
 e^{-\frac \t {2\sigma}
y^2}|y|u_0(y)\,dy}{\displaystyle \int _\R e^{-\frac \t {2\sigma}
y^2}u_0(y)\,dy}\\
 &\leq \frac {\Vert G'\Vert _\infty
} {\sqrt{2\pi\sigma}\s}\frac{\displaystyle\int _\R
 e^{-\frac {\tanh (2\sigma)} {2\sigma}
y^2}|y|u_0(y)\,dy}{\displaystyle \int _\R e^{-\frac 1 {2\sigma}
y^2}u_0(y)\,dy}=: \frac C {\s},
\end{align*}
for all $t\geq 1$. This proves \eqref{deviation}.
\end{proof}

\section{The case $f(x)=x^2$: systematic extinction in finite time 
}\label{s:plusquadratic}

The case $f(x)=x^2$ can be handled as explained in
\cite{Alf-Car}. Details are presented below: for any initial datum, the
solution becomes extinct in finite time. 

%The situation is worse than in the $f(x)=x$ case since, for any initial data, there is extinction in finite time. What happens is that, both the right and left tails quickly enlarge, so that, in order to conserve the mass, the middle part is quickly decreasing. Then the nonlocal term $\int_\R x^2u(t,x)dx$ becomes infinite very quickly for tails which are not very light (as in the $f(x)=x$ case) but also for Gaussian tails (worse than in the $f(x)=x$ case). Hence the equation becomes meaningless at a time $<\frac{\pi}{4\sigma}$ depending on the tail.

Indeed, it will turn out that there are two limitations for the time
interval of existence of the solution. The first limitation arises
when reducing equation \eqref{eqv-bis} to the heat equation
\eqref{eqw-bis} through the relation \eqref{v-w-bis}, which requires 
$$
0<t<T^{\rm Heat}:=\frac{\pi}{4\sigma}.
$$
The other limitation appears when reducing \eqref{eq-bio} to
\eqref{eqv-bis}, which requires 
$$\int_\R
e^{\frac{\tbis}{2\sigma}y^{2}}u_0(y)dy$$
to remain finite (otherwise
the solution becomes extinct). Hence, for $u_0\ge 0$  with $\int _\R
u_0=1$, we define 
\begin{equation}\label{def:T}
  T: = \sup\left\{ 0\leq t <T^{\rm Heat},\quad \int_\R  e^{\frac{\tbis}{2\sigma}y^{2}}u_0(y)dy<\infty\right\}\in [0,T^{\rm Heat}].
\end{equation}
Some typical situations are the following: if $u_0(x)$ has algebraic
or exponential tails then $T=0$ (immediate extinction); if $u_0(x)$
has Gaussian tails then $0<T<T^{\rm Heat}$ (rapid extinction in finite
time);  last, if $u_0(x)$ is compactly supported or has \lq\lq very
very light tails" then $T=T^{\rm Heat}$ (extinction in finite time).  

%For very very light tails (like $e^{-x^{4}}$ or compactly supported...), it turns out that there is  a universal extinction time which is $\frac{\pi}{4\sigma}$. To understand this, notice the presence of coefficients like $\sqrt{\tbis}$ which are defined for all $0<t<\frac{\pi}{4\sigma}$, which {\it a priori} bounds the time interval of existence of solutions.

\subsection{Results}

\begin{theo}[The solution explicitly]
\label{th:explicit-sol-bis} Let $u_0\ge 0$, with $\int _\R u_0=1$.  As long
as $\overline f (t)$ is finite, the solution of \eqref{eq-bio} with
initial data $u_0$ is given by
\begin{align}
 u(t,x)&=\frac 1{\sqrt{2\pi\sigma \tbis}}
 \frac{e^{\frac \tbis {2\sigma} x^2}
\displaystyle\int _\R e^{-\frac 1{2\sigma\tbis} \left(\frac x \cbis
-y\right)^2} u_0(y)\,dy} {\displaystyle \int _\R
e^{\frac \tbis {2\sigma} y^2}u_0(y)\,dy}\label{formula-bis}\\
&= \frac 1{\sqrt{2\pi\sigma \tbis}}
 \frac{
\displaystyle\int _\R e^{-\frac 1{2\sigma \tbis} \left(x - \frac y
\cbis\right)^2} e^{\frac \tbis {2\sigma} y^2}u_0(y)\,dy} {\displaystyle
\int _\R e^{\frac \tbis {2\sigma} y^2}u_0(y)\,dy}\,.\label{formula-2-bis}
\end{align}
As long as it exists, $\overline f (t)$ is given by
\begin{equation}\label{u-bar-bis}
\overline f (t)=\sigma \tbis+\frac{1}{(\cos (2\sigma t))^{2}}\frac{\displaystyle \int _\R
e^{\frac{\tbis}{2\sigma}y^{2}} y^{2}\, u_0(y)\,dy}{\displaystyle \int _\R e^{\frac{\tbis}{2\sigma}y^{2}} u_0(y)\,dy}.
\end{equation}
\end{theo}

\begin{rem}
 Formally, one can notice that $-x^2$ is turned into $+x^2$ in
\eqref{eq-bio} if one changes $\sigma$ to $i\sigma$, and $t$ to $-t$. After
such transforms, \eqref{formula}-\eqref{formula-2} becomes \eqref{formula-bis}-\eqref{formula-2-bis}. 
\end{rem}

\begin{prop}[Propagation of Gaussian initial
data]\label{prop:prop-gauss-bis} If
\begin{equation}\label{initial-gauss-bis}
u_0(x)=\sqrt{\frac a{2\pi}}e^{-\frac a 2(x-m)^2},\quad a>0, \quad
m\in \R,
\end{equation}
then the solution of \eqref{eq-bio} remains  Gaussian for 
$$
0<t<T=\frac{\arctan (a\sigma)}{2\sigma},
$$
and is given by
\begin{equation}\label{sol-gauss-cond-ini-bis}
u(t,x)=\sqrt{\frac {a(t)}{2\pi}}e^{-\frac{a(t)}2(x-m(t))^2},
\end{equation} where
\begin{equation}\label{sol-gauss-cond-ini2-bis}
a(t):=\frac {a\sigma-\tbis}{\sigma(1+a\sigma\tbis)}, \quad
m(t):=\frac{ma\sigma}{a\sigma\cbis-\sbis}.
\end{equation}
\end{prop}

Notice that $T<T^{\rm Heat}=\frac{\pi}{4\sigma}$. Since $a(t)\searrow 0$  as $t\nearrow T$,
it is easily seen that $u(t,x)\to 0$ uniformly in $x\in \R$. This extinction in finite time is actually true for {\it all} initial data, as stated
in the following theorem.

\begin{theo}[Extinction in finite time]\label{th:global-bis}
Let $u_0\ge 0$, with $\int _\R u_0=1$. Let $T$ be given by \eqref{def:T}.
\begin{itemize}
 \item [$(i)$]  If $T=T^{\rm Heat}$, then in \eqref{eq-bio},
 both $u(t,x)$ and $\overline  f(t)$ exist on $[0,T^{\rm Heat})$. Typically, $u \in
L^\infty_{\rm loc}((0,T^{\rm Heat})\times \R)$, $\overline f \in L^\infty_{\rm
loc}(0,T^{\rm Heat})$, and $\int_\R u(t,x)dx=1$ for all $0\leq t<T^{\rm Heat}$. Moreover, extinction at time $T^{\rm Heat}$ occurs, that is
\begin{equation*}
  u(t,x)=0,\quad \forall t> T^{\rm Heat}, \ \forall x\in \R.
\end{equation*}
\item [$(ii)$] 
 If $0<T< T^{\rm Heat}$, then extinction in finite time  occurs:
\begin{equation*}
  u(t,x)=0,\quad \forall t> T, \ \forall x\in \R.
\end{equation*}
\item [$(iii)$] If $T=0$, then $u(t,x)$ is defined for no $t>0$.
\end{itemize}
\end{theo}

\subsection{Proofs}

\begin{proof}[Proof of Theorem \ref{th:explicit-sol-bis}]
Like in the previous section, we can reduce \eqref{eq-bio} to the
heat equation by combining two changes of unknown function. First, we have
\begin{equation}\label{u-v-bis}
u(t,x)=\frac {v(t,x)}{1+\displaystyle\int _0 ^t\int _\R
x^2v(s,x)\,dxds},
\end{equation}
where $v(t,x)$ solves the Cauchy problem
\begin{equation}\label{eqv-bis}
\partial _t v = \sigma ^2\partial _{xx}v +x^2v, \quad t>0,\; x\in \R;\quad v(0,x)=u_0(x).
\end{equation}
Notice that relation \eqref{u-v-bis} is valid as  long as  $\overline f (t)$ remains finite. Next, we have
\begin{equation}\label{v-w-bis}
  v(t,x) =\frac
  1{\sqrt{\cbis}}e^{\frac{\tbis}{2\sigma}x^2}w\left(\frac{\tbis}{2\sigma},\frac
  x{\sigma\cbis}\right),
\end{equation}
where $w(t,x)$ solves the heat equation
\begin{equation}\label{eqw-bis}
   \partial _t w = \partial _{xx}w, \quad t>0,\; x\in \R;\quad w(0,x) =u_0(\sigma x).
\end{equation}
Notice that relation \eqref{v-w-bis} is valid for
$0<t<T^{\rm Heat}=\frac{\pi}{4\sigma}$.
Combining \eqref{u-v-bis}, \eqref{v-w-bis} and the integral expression of
$w$ via the heat kernel, we end up with
\begin{equation}\label{formula-ugly-bis}
u(t,x)= \frac{\sqrt{\frac \sigma{2\pi}}\frac 1{\sqrt{\sbis}}e^{\frac
\tbis{2\sigma}x^2}\displaystyle\int_ \R e^{-\frac \sigma {2\tbis}
\left(\frac x{\sigma \cbis}-y\right)^2}u_0(\sigma y)\,dy}{1+I(t) },
\end{equation}
where
$$
I(t):=\int _0^t \int _\R \int _ \R x^2 \sqrt{\frac
\sigma{2\pi}}\frac 1{\sqrt{\ssbis}}e^{\frac \tsbis{2\sigma}x^2}
e^{-\frac \sigma {2\tsbis} \left(\frac x{\sigma
\csbis}-y\right)^2}u_0(\sigma y)\,dydxds.
$$
Let us compute $I(t)$. Using Fubini's theorem, we first compute
the integral with respect to $x$. Using elementary algebra
(canonical form in particular), we get
\begin{align}
&\int _ \R x^2 e^{\frac \tsbis{2\sigma}x^2} e^{-\frac \sigma {2\tsbis}
\left(\frac x{\sigma
\csbis}-y\right)^2}\,dx\nonumber \\
&= e^{\frac{\sigma \tsbis}2 y^2} \int _\R x^2e^{-\frac 1 {2\sigma \tsbis}\left(x-\frac {\sigma y }\csbis\right)^2}dx\nonumber  \\
&= e^{\frac{\sigma \tsbis}2 y^2} \sqrt{ 2 \pi \sigma
\tsbis}\left(\sigma \tsbis +\frac{\sigma ^2 y^2}{(\csbis) ^2}\right)\label{same}
\end{align}
where we have used $\int _\R z^2e^{-\frac a 2
(z-\theta)^2}dz=\frac 1 a \sqrt {\frac {2 \pi}a}+\theta ^2
\sqrt{\frac{2\pi}a}$. Next, we pursue the computation of $I(t)$
and, integrating with respect to $s$, we find
\begin{align*}
&\int _0 ^t \sqrt{\frac \sigma{2\pi}}\frac
1{\sqrt{\ssbis}}e^{\frac{\sigma \tsbis}2 y^2} \sqrt{ 2 \pi \sigma
\tsbis}\left(\sigma \tsbis +\frac{\sigma ^2 y^2}{(\csbis) ^2}\right)\,ds\\
&=\sigma \int _0 ^t e^{-\frac 1 2 \ln (\csbis)}e^{\frac{\sigma
\tsbis}2 y^2}\left(\sigma \tsbis +\frac{\sigma ^2 y^2}{(\csbis)
^2}\right)\,ds\\
&=\sigma \int _0 ^t \frac d{ds} \left( e^{-\frac 1 2 \ln
(\csbis)}e^{\frac{\sigma \tsbis}2 y^2} \right)\,ds\\
&=\sigma\left(\frac 1 {\sqrt {\cbis}} e^{\frac{\sigma \tbis}2
y^2}-1\right).
\end{align*}
Finally, we integrate with respect to $y$ and, using $\int _\R u_0=1$,
get
\begin{align*}
I(t)&=\int _\R \sigma \left(\frac 1 {\sqrt {\cbis}}
e^{\frac{\sigma \tbis}2 y^2}-1\right)u_0(\sigma y)\,dy\\
&=\frac 1{\sqrt {\cbis}} \int _\R e^{\frac{\tbis}{2\sigma}z^2}
u_0(z)\,dz-1.
\end{align*}
Plugging this in the denominator of \eqref{formula-ugly-bis}, and
using the change of variable $z=\sigma y$ in the numerator of
\eqref{formula-ugly-bis}, we get \eqref{formula-bis}, from which
\eqref{formula-2-bis} easily follows.  Using \eqref{formula-2-bis},
Fubini theorem and the same computation as in \eqref{same}, we obtain
\eqref{u-bar-bis}.  Theorem~\ref{th:explicit-sol-bis} is proved.
\end{proof}

\begin{proof}[Proof of Proposition \ref{prop:prop-gauss-bis}] The
  proof is rather similar to that of Proposition
  \ref{prop:prop-gauss}. It consists in plugging 
the Gaussian data \eqref{initial-gauss-bis} into formula 
\eqref{formula-bis} and using elementary algebra (canonical
form). Details are omitted. 
\end{proof}

\begin{proof}[Proof of Theorem \ref{th:global-bis}] 
Let us assume $T=T^{\rm Heat}$ and prove  $(i)$. Since 
$$\int_\R
e^{\frac{\tbis}{2\sigma}y^{2}}u_0(y)dy <\infty \text{ for all }
0<t<T^{\rm Heat},$$
we have $\int _\R e^{\frac{\tbis}{2\sigma}y^{2}} y^{2}u_0(y)\,dy
<\infty$ for all $0<t<T^{\rm Heat}$, and therefore both
\eqref{u-bar-bis} and 
\eqref{formula-2-bis} are meaningful for all $0<t<T^{\rm Heat}$. It
follows from \eqref{formula-2-bis} that 
$$
0\leq u(t,x)\leq 
 \frac{1}{\sqrt{2\pi\sigma \tbis}},
$$
and the right hand side goes to zero as $t\nearrow T^{\rm Heat}$.

Let us  assume $0<T< T^{\rm Heat}$ and prove $(ii)$. It follows from \eqref{formula-bis} that
$$
0\leq u(t,x)\leq 
 \frac{e^{\frac \tbis {2\sigma} x^2}}
{\sqrt{2\pi\sigma \tbis}\displaystyle \int _\R
e^{\frac \tbis {2\sigma} y^2}u_0(y)\,dy},
$$
and, the right hand side goes to zero as $t\nearrow T<T^{\rm Heat}$.

Finally,  assume $T=0$ and prove $(iii)$. Supposing by contradiction
that there 
  is a $\tau>0$ such that $\overline f$ is finite on $[0,\tau]$, then
  \eqref{formula-bis} would hold true. On the other hand, the
  assumption $T=0$, along with \eqref{formula-bis}, would imply
  $u(t,x)=0$ for all $t\in (0,\tau]$ and all $x\in \R$, while we know that so long as $\overline f$ is finite, we have
  $\int_\R u(t,x)dx=1$, hence a contradiction.
\end{proof}

\bibliographystyle{amsplain}    %ordre alphab�tique, noms en majuscule
\bibliography{biblio}

\end{document}